\documentclass{amsart}

\usepackage{amssymb}

\usepackage[T1]{fontenc}
\usepackage{bbm}
\usepackage{graphicx}
\usepackage{latexsym}
\usepackage{enumerate}
\usepackage{graphicx}

\newtheorem{theorem}{Theorem}[section]
\newtheorem{lemma}[theorem]{Lemma}

\theoremstyle{definition}

\theoremstyle{remark}

\newcommand\E{\mathbb{E}}
\newcommand\R{\mathbb{R}}

\numberwithin{equation}{section}

\begin{document}
\title[Weighted inequalities]{On the weighted weak-type constant\\ for the dyadic square function}

\author{Adam Os\k{e}kowski}
\address{Faculty of Mathematics, Informatics and Mechanics, University of Warsaw,
Banacha 2, 02-097 Warsaw, Poland}
\email{A.Osekowski@mimuw.edu.pl}


\begin{abstract}
We present an explicit construction of examples showing that the
estimate  
$ \|S(f)\|_{L^2(w)\to L^{2,\infty}(w)}\lesssim \big([w]_{A_2} 
\log(1 + [w]_{A_2})\big)^{1/2}$  for the dyadic square function is
sharp in terms of the characteristic $[w]_{A_2}$.
\end{abstract}

\maketitle

\section{Introduction}
The purpose of the paper is to address an open problem in the one-weight theory, arising in the context of weak-type $(2,2)$ estimates for square functions. To present the result from an appropriate perspective, we will first discuss related statements for maximal and singular integral operators.

\smallskip

Let $M$ denote the Hardy--Littlewood maximal operator on $\mathbb{R}^d$, defined by
\[
Mf(x)=\sup_{Q\ni x}\frac{1}{|Q|}\int_Q |f(y)|\,\mathrm{d}y,
\]
where the supremum is taken over all cubes $Q\subset\mathbb{R}^d$ containing $x$ and having sides parallel to the coordinate axes. In 1971, Fefferman and Stein \cite{F} established the weak-type $(1,1)$ estimate
\begin{equation}\label{FS}
\|Mf\|_{L^{1,\infty}(w)}
\lesssim_d \|f\|_{L^1(Mw)}.
\end{equation}
Here $w$ is an arbitrary weight on $\mathbb{R}^d$, that is, a nonnegative measurable function; any such weight can be identified with the Borel measure \(w(E)=\int_E w(x)\,\mathrm{d}x \). Furthermore, for $1\leq p<\infty$, $\|f\|_{L^p(w)}$ and $\|f\|_{L^{p,\infty}(w)}$ denote the associated strong and weak $L^p$ norms, given by $\left(\int_{\R^d}|f|^pw\right)^{1/p}$ and $\sup_{\lambda>0}(\lambda w(|f|\geq \lambda))^{1/p}$, respectively.

\smallskip

Fefferman and Stein used \eqref{FS} to obtain a class of sharp estimates for vector-valued maximal operators. A few years later, Muckenhoupt and Wheeden conjectured that the estimate \eqref{FS} should remain valid if the maximal operator on the left-hand side is replaced by an arbitrary Calder\'on--Zygmund operator $T$. More precisely, they conjectured that
\begin{equation}\label{MW}
\|Tf\|_{L^{1,\infty}(w)}\lesssim_{d,T}\|f\|_{L^1(Mw)},
\end{equation}
where $C_{d,T}$ is a finite constant depending only on the dimension and on the operator $T$. This problem remained open for almost forty years. In 2012, Reguera and Thiele \cite{RT} disproved it, by showing that the estimate fails for the Hilbert transform; an earlier result of Reguera \cite{R} showed that the corresponding dyadic version of the conjecture is false as well.

A weaker form of the Muckenhoupt--Wheeden conjecture has also attracted considerable attention. Recall that a weight $w$ on $\mathbb{R}^d$ belongs to the Muckenhoupt class $A_1$ if there exists a constant $c>0$ such that
\begin{equation}\label{A1}
Mw\leq cw
\qquad\text{a.e. on }\mathbb{R}^d.
\end{equation}
The smallest admissible constant is denoted by $[w]_{A_1}$ and is called the $A_1$ characteristic of $w$. If \eqref{MW} were true, then, for every $A_1$ weight $w$, we would have
\[
\|Tf\|_{L^{1,\infty}(w)} \lesssim_{d,T} [w]_{A_1} \|f\|_{L^1(w)}.
\]
Indeed, this follows immediately from the pointwise inequality
$Mw\leq[w]_{A_1}w$. However, this weaker statement is also false, even for the Hilbert transform, and an analogous negative result holds in the dyadic setting; see \cite{NRVV,NRVV2}.

On the positive side, Lerner et al. \cite{LOP} established the following estimate: for every Calder\'on--Zygmund operator $T$,
\begin{equation}\label{sss}
\|Tf\|_{L^{1,\infty}(w)} \lesssim_{d,T} [w]_{A_1}
\log\big(1+[w]_{A_1}\big) \|f\|_{L^1(w)}.
\end{equation}
Furthermore, it turns out that the $L\log L$ dependence on the $A_1$ characteristic in \eqref{sss} is optimal (cf. \cite{LNO}). For the sake of completeness,  let us mention the versions concerning weighted weak-type $(p,p)$ bounds, for $1<p<\infty$: as shown by Hyt\"onen et al. \cite{H0}, we have
$$ \|Tf\|_{L^{p,\infty}(w)} \lesssim_{d,T}
[w]_{A_p} \|f\|_{L^p(w)}$$
and the linear dependence on the characteristics is optimal. Here the $A_p$ characteristics is given by
\begin{equation}\label{Ap}
 [w]_{A_p}=\sup \left(\frac{1}{|Q|}\int_Q w\right)\left(\frac{1}{|Q|}\int_Q w^{1/(1-p)}\right)^{p-1},
\end{equation}
where the supremum is taken over all cubes $Q$ contained in $\R^d$, with sides parallel to the axes.

\smallskip

Analogous questions can be posed for other classes of operators. In this paper, we focus on the dyadic square function on $[0,1)$, for which these questions have received considerable attention in the literature. Let $\mathcal{D}$ denote the standard dyadic lattice of intervals on $[0,1)$, organized into generations
$$\mbox{$\mathcal{D}_0=\{[0,1)\}, \quad \mathcal{D}_1=\{[0,\frac{1}{2}),[\frac{1}{2},1)\}, \quad  \mathcal{D}_2=\big\{[0,\frac{1}{4}),[\frac{1}{4},\frac{1}{2}),[\frac{1}{2},\frac{3}{4}),[\frac{3}{4},1)\big\},\, \ldots$}.$$
 For an integrable function $f$ on $[0,1)$, let $(f_n)_{n\geq 0}$ be the associated dyadic martingale, given by 
$$ f_n=\mathbb{E}\big(f\,|\,\sigma(\mathcal{D}_n)\big)=\sum_{I\in \mathcal{D}_n} \chi_I\frac{1}{|I|}\int_I f,\qquad n=0,\,1,\,2,\,\ldots.$$ Then the dyadic square function of $f$ is defined by
\[
S f(x)
=
\left( |f_0|^2+
\sum_{n\geq 0} |f_{n+1}-f_n|^2
\right)^{1/2}.
\]
For a fixed $1\leq p<\infty$, we define the dyadic Muckenhoupt classes
$A_p^d$ on $[0,1)$ by suitably modifying the definitions in
\eqref{A1} and \eqref{Ap}. More precisely, a positive integrable
function $w$ belongs to $A_1^d$ if
\[
M^d w \leq c w
\qquad\text{almost everywhere on }[0,1),
\]
where $M^d$ denotes the dyadic maximal operator on $[0,1)$. The smallest admissible constant $c$ is called the $A_1^d$ characteristic
of $w$ and is denoted by $[w]_{A_1^d}$. For $1<p<\infty$, we say that a positive integrable function $w$ belongs to $A_p^d$ if the quantity
in \eqref{Ap}, with the supremum taken over all dyadic subintervals $Q\subseteq[0,1)$, is finite.

\smallskip

Now, one can ask about the weighted weak-type bounds for $S$, with the sharp dependence on the $A_p$ characteristics of the involved weight. The following result was obtained by Lacey and Scurry \cite{LS} (for $p<2$) and H\"ytonen-Li \cite{HL} (for $p\neq 2$). 
\begin{theorem}
Suppose that $1\leq p<\infty$, $p\neq 2$. Then we have
$$ \|S(f)\|_{L^{p,\infty}(w)}\lesssim [w]_{A_p}^{\max\{1/p,1/2\}} \|f\|_{L^p(w)}.$$
The exponent $\max\{1/p,1/2\}$ is the best possible.
\end{theorem}

It turns out that for $p=2$ the situation is significantly harder. Lacey and Scurry \cite{LS} established the estimate
$$ \|S(f)\|_{L^{2,\infty}(w)}\lesssim [w]_{A_2}^{1/2}\log(1+[w]_{A_2})\|f\|_{L^2(w)},$$ 
which was later improved, by Domingo-Salazar, Lacey and Rey \cite{DS}, to the form
$$ \|S(f)\|_{L^{2,\infty}(w)}\lesssim \big([w]_{A_2}\log(1+[w]_{A_2})\big)^{1/2}\|f\|_{L^2(w)}.$$
There is a very natural question whether the additional logarithmic factor is necessary. This problem has gained a lot of interest in the literature, and the primary goal of this paper is to answer this question in the affirmative. Here is our main result.

\begin{theorem}\label{mainthm}
For any $c\geq 1$ there is a weight $w\in A_2^d$, satisfying $[w]_{A_2}\geq c$, with the following property: for some function $f\in L^2(w)$ we have
$$ \|S(f)\|_{L^{2,\infty}(w)}> \frac{e^{-2}}{48} \big([w]_{A_2}\log(1+[w]_{A_2})\big)^{1/2}\|f\|_{L^2(w)}.$$
\end{theorem}

This completes the description of sharp weighted weak-type estimates for square functions. The resulting dependence on the weight characteristic is closely analogous to that for singular integral operators, except that the critical exponent, at which a logarithmic correction appears, is shifted from $p=1$ to $p=2$.

Theorem \ref{mainthm} will be established in the next section, by constructing an appropriate example. It will be convenient to use probabilistic language: the construction is more naturally formulated in martingale terms, with the relevant functions and weights viewed as stochastic processes satisfying prescribed evolution rules.

\section{Construction}

This section is divided into two parts.

\subsection{A building block} 
Let $N\geq 2$ be a fixed positive integer. We start with the introduction of three auxiliary parameters. First, let $\alpha$ be a positive number satisfying the equation
\begin{equation}\label{defalpha}
 \alpha\left(\frac{2}{3}\cdot 2^N+\frac{1}{3}\cdot 2^{-N}\right)-\alpha^{-1}-\frac{1}{3}(2^N-2^{-N})=0.
\end{equation}
Of course, $\alpha$ can be computed explicitly; however, all that we will need below is that $\alpha$ is unique and satisfies $1/2<\alpha<1/2+3\cdot 2^{-N}$. To see this, denote the left-hand side of \eqref{defalpha} by $H(\alpha)$. Then $H$ is strictly increasing on $(0,\infty)$ and satisfies 
$$ H\left(\frac{1}{2}\right)=2^{-1-N}-2<0,\qquad  H\left(\frac{1}{2}+\frac{3}{2^N}\right)> \left(\frac{1}{2}+\frac{3}{2^N}\right)\cdot \frac{2}{3}\cdot 2^N-2-\frac{2^N}{3}=0.$$
The remaining two parameters $\beta,\gamma\in (0,1)$ are given by
\begin{equation}\label{defbeta}
\beta=(1+2^{-N}\alpha)^{-1/2},\qquad \gamma=(1+2^{-N}\alpha^{-1})^{-1/2}.
\end{equation}
Next, we define two sequences of positive numbers which will be useful later. For any $n=0,\,1,\,2,\,\ldots,\,N$, let
$$ x_n=\frac{1}{3}\cdot 2^{N-n}-\frac{1}{3}\cdot 2^{n-N}+\alpha^{-1},\qquad y_n=\frac{4}{3}\cdot 2^n+\frac{2}{3}\cdot 2^{-n}.$$
Some important properties of these sequences are gathered in a lemma below.

\begin{lemma}
(i) We have $x_N=\alpha^{-1}$, $y_0=2$ and 
\begin{equation}\label{recurrences}
 x_n=\frac{x_{n+1}+2^{N-n-1}+\alpha^{-1}}{2},\qquad y_n=\frac{y_{n+1}+2^{-n}}{2}
\end{equation}
for $n=0,\,1,\,2,\,\ldots,\,N-1$. 

(ii) We have the identity
\begin{equation}\label{defalfa}
x_0(1+\alpha)=(2^N+\alpha^{-1})\alpha.
\end{equation}

(iii) For any $n=0,\,1,\,2,\,\ldots,\,N$ we have
\begin{equation}\label{char}
 x_n(2^n+\alpha)<6\cdot 2^N.
\end{equation}
\end{lemma}
\begin{proof}
The property (i) is straightforward, the identity in (ii) is equivalent to \eqref{defalpha}. To check \eqref{char}, we use the inclusion $\alpha\in [1/2,1]$, which gives 
$x_n<2^{N-n}+2\leq 3\cdot 2^{N-n}$ and $2^n+\alpha<2\cdot 2^n$.
\end{proof}

We proceed to the construction of the main building block. Let us assume that the probability space is some dyadic subinterval $I$ of $[0,1)$,  equipped with its Borel subsets and normalized Lebesgue measure; we equip it with the filtration induced by consecutive dyadic generations. 
Consider the three-dimensional dyadic martingale $\big((u_n,v_n,\varphi_n)\big)_{n\geq 0}$, whose distribution (evolution) is uniquely determined by the following requirements.

\medskip

(i) We have $(u_0,v_0,\varphi_0)\equiv (x_0,1+\alpha,y_0)$ almost surely.

\smallskip

(ii) For $n=0,\,1,\,2,\,\ldots,\,N-1$, any point of the form $(x_n,2^n+\alpha,y_n)$ leads to the state $(x_{n+1},2^{n+1}+\alpha,y_{n+1})$ or $(2^{N-n-1}+\alpha^{-1},\alpha,2^{-n})$.

\smallskip

(iii) The point $(x_N,2^N+\alpha,y_N)=(\alpha^{-1},2^N+\alpha,y_N)$ leads to one of the positions $\Big((1\pm \beta)\alpha^{-1},(2^N+\alpha)(1\mp \beta),y_N(1\mp \beta)\Big).$

\smallskip

(iv) For any $n=0,\,1,\,2,\,\ldots,\,N$, any point of the form $(2^{N-n-1}+\alpha^{-1},\alpha,2^{-n})$ leads to $(\alpha^{-1},\alpha,0)$ or to $(2^{N-n}+\alpha^{-1},\alpha,2^{1-n})$.

\smallskip

(v) All the remaining points are absorbing.

\medskip

It is easy to check, using \eqref{recurrences}, that the above requirements do describe the dyadic martingale. Indeed, in (ii), (iii) and (iv), with transition probabilities equal to $1/2$, the appropriate averages match the starting positions. To gain some intuition behind the construction, let us first look at the behavior of the martingale $((u_n,v_n))_{n\geq 0}$ (see Figure \ref{graph}). The pair starts from the point 
\begin{figure}[htbp]
\begin{center}
\includegraphics[scale=0.9]{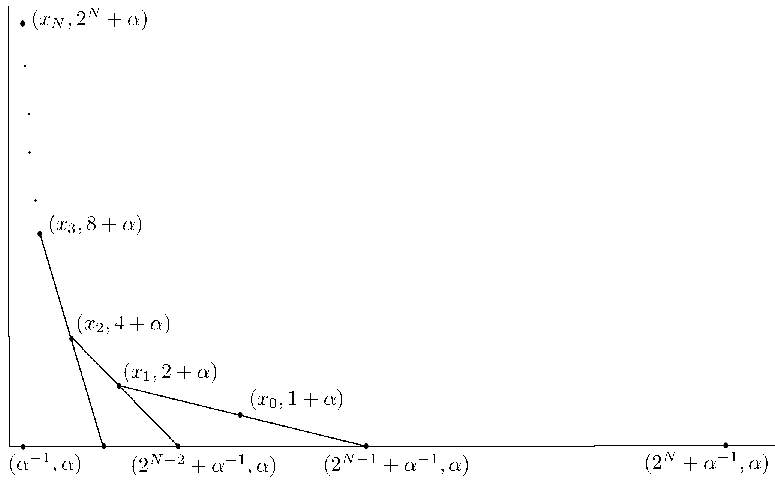}
\caption{The evolution of the pair $((u_n,v_n))_{n\geq 0}$. The process starts at $(x_0,1+\alpha)$ and moves along the endpoints of skew line segments, until it drops to the line $v=\alpha$; then it moves horizontally.}\label{graph}
\end{center}
\end{figure} 
$(x_0,1+\alpha)$ and it is convenient to split its evolution into two phases. 

\smallskip

\emph{Phase I: the requirements (ii) and (iii)}. During  this phase, the second coordinate $v$ increases for a number of steps, moving from $1+\alpha$ to $2+\alpha$, $4+\alpha$, and so on. There are two possibilities: first, it might happen that at some time $n$, the martingale $v$ drops from $2^n+\alpha$ to $\alpha$; then Phase II begins. The second possibility is that the pair $(u,v)$ reaches the point $(x_N,2^N+\alpha)=(\alpha^{-1},2^N+\alpha)$. If this happens, in the next step the pair moves to the curve $xy=1$: indeed, we have
$$ (1\pm \beta)\alpha^{-1}\cdot (2^N+\alpha)(1\mp \beta)=(1-\beta^2)(2^N\alpha^{-1}+1)=1,$$
by the very definition of $\beta$ (see \eqref{defbeta}). After that, the pair terminates, and in this case Phase II is empty.

\smallskip

\emph{Phase II: the requirements (iv)}. If $v$ drops to $\alpha$ at some time $n$, then the process $(u,v)$ starts to move horizontally (i.e., $v$ remains equal to $\alpha$ until the end of the evolution). At each step, the first coordinate either goes to $\alpha^{-1}$, or increases from the point of the form $2^k+\alpha^{-1}$ to $2^{k+1}+\alpha^{-1}$. The phase terminates when $u$ reaches the set $\{\alpha^{-1},2^N+\alpha^{-1})$.

\smallskip

Some important remarks are in order. First, it is clear that $(u,v)$ is a finite martingale: the evolution lasts $2N$ steps at the longest. Next, the pair takes values in the set $\{(x,y)\in \R_+^2\,:\,1\leq xy\leq 6\cdot 2^N\}$. Indeed, during Phase~I this follows from \eqref{char}, and for Phase II we use \eqref{defalfa} and \eqref{char} again, to observe that 
$$ 1=\alpha \cdot \alpha^{-1}\leq \alpha \cdot (2^k+\alpha^{-1})\leq \alpha \cdot (2^N+\alpha^{-1})=x_0(1+\alpha)<6\cdot 2^N.$$
Finally, observe that for most trajectories, the pair $(u,v)$ terminates at the curve $xy=1$. The exceptional trajectories correspond to the scenarios in which Phase~I lasts for $k$ steps ($k=1,\,2,\,\ldots,\,N$) and then, in Phase II, the pair moves after $k$ steps to the rightmost point $(2^N+\alpha^{-1},\alpha)$. Therefore, we may write
\begin{equation}\label{uv}
 \mathbb{P}(u_\infty v_\infty\neq 1)=\sum_{k=1}^{N} 2^{-k}\cdot 2^{-k}=\frac{1-4^{-N}}{3}. 
 \end{equation}
Furthermore, by \eqref{defalfa}, if $u_\infty v_\infty\neq 1$ (that is, $(u_\infty,v_\infty)=(2^N+\alpha^{-1},\alpha)$), then $u_\infty v_\infty=u_0v_0$. This will be crucial for the iteration argument which will be described later. 

\smallskip

Now we turn our attention to the martingale $(\varphi_n)_{n\geq 0}$; to analyze its behavior, it will be convenient to refer to Phases I and II  described above.

\smallskip

\emph{Phase I.} Here the martingale $(\varphi_n)_{n\geq 0}$ evolves in the same manner as $(v_n)_{n\geq 0}$: it increases for a number of steps and then drops to some level (different for different steps), initiating Phase II. In the exceptional case in which $(u,v)$ gets to the point $(x_N,2^N+\alpha)$, the process $\varphi$ reaches the value $y_N$, and in the next step it jumps to $y_N(1\pm \beta)$. Then the evolution stops.

\smallskip

\emph{Phase II.} Here the movement of $\varphi$ is similar to that of $u$: at each step $\varphi$ either drops to zero (and stops ultimately), or doubles its value, until it gets to $2$. The trajectories of $\varphi$ which lead to $2$ correspond precisely to those trajectories of $(u,v)$, which terminate at the position $(2^N+\alpha^{-1},\alpha)$. That is, we have the equality $\{\varphi_\infty=2\}=\{u_\infty v_\infty\neq 1\}$.

\smallskip

It follows from the above analysis that the martingale $(\varphi_n)_{n\geq 0}$ starts from $2$ and its almost sure limit $\varphi_\infty$ takes values in the set $\{0,y_N(1\pm \beta),2\}$. We have
$$ \mathbb{P}(\varphi_\infty=y_N(1-\beta))=\mathbb{P}(\varphi_\infty=y_N(1+\beta))=2^{-N-1}$$
and hence
\begin{equation}\label{K0}
\begin{split}
 &\E \varphi_\infty^2 u_\infty \chi_{\{u_\infty v_\infty=1\}}\\
&=2^{-N-1}\Big[\big(y_N(1-\beta)\big)^2(1+\beta)\alpha^{-1}+\big(y_N(1+\beta)\big)^2(1-\beta)\alpha^{-1}\Big]\\
 &\leq 2^{N+2}\alpha^{-1}(1-\beta^2)\beta\leq 8,
\end{split}
\end{equation}
where the first inequality is due to the estimate $y_N\leq 2^{N+1}$, and in the second passage we have used the bounds $\alpha^{-1}\leq 2$, $\beta\leq 1$ and $ 1-\beta^2=(2^N\alpha^{-1}+1)^{-1}\leq 2^{-N}.$

The final observation concerns the conditional behavior of the square function of $\varphi$ on $\{u_\infty v_\infty \neq 1\}$. Namely, on this set the martingale $\varphi$ increases for a number of steps, then drops, thus initiating Phase II, and finally moves up until it gets back to $2$. Formally, we can write this as the splitting
\begin{equation}\label{An}
 \{u_\infty v_\infty \neq 1\}=\bigcup_{n=1}^N A_n,\qquad A_n=\{\varphi_n<\varphi_0\leq \varphi_{n-1},\,\varphi_\infty=2\}
 \end{equation}
(the event $A_n$ corresponds to the scenario in which the drop of $\varphi$ is observed at time $n$). 
 Note that  $\mathbb{P}(A_n)=4^{-n}$. Furthermore, on the set $A_n$ we have 
$$ S(\varphi)\geq |\varphi_n-\varphi_{n-1}|= y_n-y_{n-1}=\frac{4}{3}\cdot 2^{n-1}-\frac{2}{3}\cdot 2^{-n}\geq 2^{n-1}.$$
Squaring this inequality, we see that  if we introduce the random variable
$$ \xi=\begin{cases}
4^{n-1} & \mbox{on }A_n,\,n=1,\,2,\,\ldots,\,N,\\
0 & \mbox{on }\{u_\infty v_\infty=1\},
\end{cases}$$
then we have $S^2(\varphi)\geq \xi$ and the conditional distribution of $\xi$ on $\{u_\infty v_\infty\neq 1\}$ is
\begin{equation}\label{cond_dist}
 \mathbb{P}\big(\xi=4^{n-1}\,|\,u_\infty v_\infty\neq 1\big)=4^{-n}\cdot \frac{3}{1-4^{-N}},\qquad n=1,\,2,\,\ldots,\,N.
\end{equation}

\subsection{Construction of the example} We are ready to describe the functions and weights which provide the appropriate lower bound. Note that in the previous subsections, the underlying probability space was some arbitrary dyadic subinterval $I$ of $[0,1)$. It will be convenient to denote the constructed triple by $(u^I,v^I,\varphi^I)$.

Let us define an appropriate finite dyadic martingale $(w,\sigma,f)=((w_n,\sigma_n,f_n))_{n\geq 0}$ on $\Omega=[0,1)$; this will be done by $2^N$ successive extensions of processes constructed above. 

\smallskip

\emph{Step 1.} We start by taking $(w,\sigma,f)=(u^\Omega,v^\Omega,\varphi^\Omega)$. On the set $\{u^\Omega_\infty v^\Omega_\infty=1\}$, this is the final formula; on the compliment of this set, we will need to modify (extend) the martingale. It will be convenient to say that on the set $\{u^\Omega_\infty v^\Omega_\infty=1\}$, the martingale $(w,\sigma,f)$ is modified $0$ times.

\smallskip

\emph{Step 2.} Note that  $\{u_\infty^\Omega v_\infty^\Omega\neq 1\}$ is the union of a finite number of pairwise disjoint dyadic intervals $I_1$, $I_2$, $\ldots$, $I_m$, on which $(u_\infty^\Omega,v_\infty^\Omega,\varphi_\infty^\Omega)\equiv (2^N+\alpha^{-1},\alpha,2)$. (Actually, these intervals are precisely the events $A_n$ appearing in \eqref{An} - but we will not need this). On each interval $I_j$, we consider the rescaled martingale 
$$(u^{I_j}(1+\alpha)/\alpha,v^{I_j} \alpha/(1+\alpha),\varphi^{I_j}).$$
 This new process starts from $(x_0(1+\alpha)/\alpha,\alpha,2)$, which by \eqref{defalfa} is precisely the terminal position of $(u_\infty^\Omega,v_\infty^\Omega,\varphi_\infty^\Omega)$. Therefore, we can extend the martingale $(w,\sigma,f)$ by requiring that on  $I_j$, the terminal variable $(w,\sigma,f)$ is equal to 
$$(u_\infty^{I_j}(1+\alpha)/\alpha,v_\infty^{I_j}\alpha/(1+\alpha),\varphi^{I_j}_\infty).$$
In other words, the triple $(w,\sigma,f)$ behaves as $(u^\Omega,v^\Omega,\varphi^\Omega)$ and then, when it reaches the state $(2^N+\alpha^{-1},\alpha,2)$ - which corresponds to one of the events $I_j$ - then it copies the evolution of  $(u^{I_j}(1+\alpha)/\alpha,v^{I_j}\alpha/(1+\alpha),\varphi^{I_j}).$ Now, on the set $\{u^{I_j}_\infty v^{I_j}_\infty=1\}$, this will be the end of the construction (we will say that on this set, the process was modified once). On the set $\{u^{I_j}_\infty v^{I_j}_\infty\neq 1\}$ we have $(w_\infty,\sigma_\infty,f_\infty)=((2^N+\alpha^{-1})(1+\alpha)/\alpha,\alpha^2/(1+\alpha),2)$ and the further extension is necessary.

\smallskip

\emph{Step 3.} Essentially, this step is the repetition of the previous one and describes the modification of $(w,\sigma,f)$ after $k$ iterations, where $k<2^N$. Namely, suppose that we have defined a finite martingale $(w,\sigma,f)$ such that on the set $\{w_\infty \sigma_\infty\neq 1\}$, the terminal value $(w_\infty,\sigma_\infty,f_\infty)$ is equal to $\Big((2^N+\alpha^{-1})(1+\alpha)^k/\alpha^k,\alpha^{k+1}/(1+\alpha)^k,2\Big)$. Suppose in addition, that the set $\{w_\infty \sigma_\infty\neq 1\}$ is a union of a finite number of pairwise disjoint dyadic intervals; with a slight abuse of notation, let us denote these intervals by $I_1$, $I_2$, $\ldots$, $I_m$ again. On each $I_j$, we consider the rescaled martingale $(u^{I_j}(1+\alpha)^{k+1}/\alpha^{k+1},v^{I_j} \alpha^{k+1}/(1+\alpha)^{k+1},\varphi^{I_j})$. This process starts from $(x_0(1+\alpha)^{k+1}/\alpha^{k+1},\alpha^{k+1}/(1+\alpha),2)$, which by \eqref{defalfa} is precisely the terminal value of $(w,\sigma,f)$. Thus we can extend the latter triple on $I_j$, requiring that the conditional distribution coincides with that of the rescaled process. This is precisely the $k+1$-st iteration.

\smallskip

\emph{Step 4.} Finally, we make a final modification, after iterating $2^N$ times. Precisely, suppose that we have defined a finite martingale $(w,\sigma,f)$ such that on the set $\{w_\infty \sigma_\infty\neq 1\}$, the terminal value $(w_\infty,\sigma_\infty,f_\infty)$ is equal to $\Big((2^N+\alpha^{-1})(1+\alpha)^{2^N}/\alpha^{2^N},\alpha^{{2^N}+1}/(1+\alpha)^{2^N},2\Big)$. Let us further assume that the set $\{w_\infty \sigma_\infty\neq 1\}$ is a union of a finite family of pairwise disjoint intervals $I_1$, $I_2$, $\ldots$, $I_m$. Then on each $I_j$, we assume that the martingale $(w,\sigma,f)$ leads to one of the points
$$ \Big((1\pm \gamma)(2^N+\alpha^{-1})(1+\alpha)^{2^N}/\alpha^{2^N},(1\mp \gamma)\alpha^{{2^N}+1}/(1+\alpha)^{2^N},2(1\mp \gamma)\Big)$$
with probabilities $1/2$. Here $\gamma$, given by \eqref{defbeta}, is chosen so that $w_\infty\sigma_\infty=1$.

\smallskip

This completes the description of $(w,\sigma,f)$. From now on, following the usual convention in semimartingale theory, we will identify $w,\,\sigma$ and $f$ with the terminal values $w_\infty$, $\sigma_\infty$ and $f_\infty$; this should not lead to any confusion. Directly from the construction, we see that the triple is a finite dyadic martingale satisfying $w \sigma=1$ almost surely; that is, we have $\sigma=w^{-1}$ with probability $1$.

\begin{lemma}
We have $\frac{1}{2}\cdot 2^N\leq [w]_{A_2}\leq 6\cdot 2^N$.
\end{lemma}
\begin{proof}
 As we have discussed in the previous subsection, for each $I$ the pair $(u^I,v^I)$ takes values in the set $\{(x,y)\in \R_+^2\,:\,1\leq xy\leq 6\cdot 2^N\}$. The scaling of the form $(u^I,v^I)\mapsto (u^I\cdot \lambda, v^I\cdot \lambda^{-1})$ preserves this property and hence the pair $(w,\sigma)$ also enjoys this condition. In other words, for any dyadic interval $I$ we have the double inequality
$$ 1\leq \frac{1}{|I|}\int_I w \cdot \frac{1}{|I|}\int_I \sigma \leq 6\cdot 2^N,$$
which combined with the equality $\sigma=w^{-1}$ implies that $w$ is a dyadic $A_2$ weight satisfying $[w]_{A_2}\leq 6\cdot 2^N$. Finally, note that
$$ [w]_{A_2}\geq \int_{[0,1)}w\cdot \int_{[0,1)} \sigma=x_0(1+\alpha)=(2^N+\alpha^{-1})\alpha\geq 2^N\cdot \frac{1}{2}. \qedhere$$
\end{proof}

\begin{lemma}
We have $\|f\|_{L^2(w)}\leq 2^{(N+3)/2}$.
\end{lemma}
\begin{proof}
To compute $ \E f_\infty^2w_\infty$, let us distinguish the set $J_k=\{(w,\sigma,f)$ was modified (extended) $k$ times$\}$. As we have proved in \eqref{K0}, we have 
$$ \E f_\infty^2w_\infty \chi_{J_0}\leq 8.$$
Therefore, by \eqref{uv} and the scaling described in Step 2 above, we get
$$ \E f_\infty^2 w_\infty \chi_{J_k}\leq \left(\frac{1-4^{-N}}{3}\right)^k \cdot \left(\frac{1+\alpha}{\alpha}\right)^k\cdot 8\leq 8$$
for $1\leq k<{2^N}$ (here in the last passage we have used the estimate $\alpha> 1/2$). 
Finally, exploiting the information from Step 4 above, we compute directly that
\begin{align*}
 \E f_\infty^2 w_\infty \chi_{J_{2^N}}&=\left(\frac{1-4^{-N}}{3}\right)^{2^N} \cdot \frac{1}{2}\bigg[\big(2(1-\gamma)\big)^2\cdot (1+ \gamma)(2^N+\alpha^{-1})\left(\frac{1+\alpha}{\alpha}\right)^{2^N} \\
 &\qquad \qquad \qquad \qquad +\big(2(1+\gamma)\big)^2\cdot (1- \gamma)(2^N+\alpha^{-1})\left(\frac{1+\alpha}{\alpha}\right)^{2^N} \bigg]\\
 &=4\left(\frac{1-4^{-N}}{3}\right)^{2^N} \left(\frac{1+\alpha}{\alpha}\right)^{2^N} (2^N+\alpha^{-1})(1-\gamma^2)\gamma \leq 8,
\end{align*}
since $\gamma\leq 1$, $(1-\gamma^2)(2^N+\alpha^{-1})=\alpha^{-1}\leq 2$ and
$$ \left(\frac{1-4^{-N}}{3}\right)^{2^N} \left(\frac{1+\alpha}{\alpha}\right)^{2^N}\leq \left(\frac{1+\alpha}{3\alpha}\right)^{2^N}\leq 1$$
(again, the latter bound is due to $\alpha> 1/2$). Therefore, we have proved that $ \E f^2w\leq 8\cdot {2^N}$, which is precisely the claim.
\end{proof}

\begin{lemma}
We have $w(S^2(f)\geq N\cdot 2^{N-2})\geq \frac{1}{12}e^{-4}\cdot 2^N$ for sufficiently large $N$.
\end{lemma}
\begin{proof}
Let us ignore the final move of the martingale $f$ on the set $J_{2^N}$, described in Step 4 above; denoting the corresponding ``truncated'' square function by $\tilde{S}(f)$, we have $S(f)\geq \tilde{S}(f)$ and therefore
\begin{align*}
 w(S^2(f)\geq N\cdot 2^{N-2})&\geq w\big(\{\tilde{S}^2(f)\geq N\cdot 2^{N-2}\}\cap J_{2^N}\big)\\
 &=(2^N+\alpha^{-1})\left(\frac{1+\alpha}{\alpha}\right)^{{2^N}}\mathbb{P}\big(\{\tilde{S}^2(f)\geq N\cdot 2^{N-2}\}\cap J_{2^N}\big)\\
 &>2^N\left(\frac{1+\alpha}{\alpha}\right)^{2^N}\mathbb{P}(J_{2^N})\mathbb{P}\big(\tilde{S}^2(f)\geq N\cdot 2^{N-2}| J_{2^N}\big)\\
 &=2^N\left(\frac{1+\alpha}{\alpha}\right)^{2^N}\left(\frac{1-4^{-N}}{3}\right)^{2^N}\mathbb{P}\big(\tilde{S}^2(f)\geq N\cdot 2^{N-2}| J_{2^N}\big).
\end{align*}
It will be proved in Lemma \ref{prob} below that for sufficiently large $N$ we have 
\begin{equation}\label{limit}
 \mathbb{P}(\tilde{S}^2(f)\geq N2^{N-2}|J_{2^N})\geq \frac{1}{3}.
\end{equation}
Assuming the validity of this estimate, we thus obtain
$$ w(S^2(f)\geq N2^{N-2})\geq \frac{2^N}{3} (1-4^{-N})^{2^N}\cdot \left(\frac{1+\alpha}{3\alpha}\right)^{2^N}.$$
Now, if $N$ is large enough, we have
$$ (1-4^{-N})^{2^N}\geq \frac{1}{2}.$$
Furthermore, recalling that $1/2<\alpha<1/2+3\cdot 2^{-N}$, we may write
$$ \left(\frac{1+\alpha}{3\alpha}\right)^{2^N}=\left(1+\frac{1-2\alpha}{3\alpha}\right)^{2^N}\geq \left(1-\frac{2^{1-N}}{\alpha}\right)^{2^N}\geq (1-2^{2-N})^{2^N}$$
and the latter quantity is bigger than $e^{-4}/2$ for sufficiently large $N$. Putting all the above facts together, we get the claim.
\end{proof}

The three lemmas above yield our main result.

\begin{proof}[Proof of Theorem \ref{mainthm}]
This is immediate: we have $[w]_{A_2}\leq 6\cdot 2^N$ and $\ln(1+ [w]_{A_2})< N$ (for large $N$) and hence
$$ \frac{\|S(f)\|_{L^{2,\infty}(w)}}{\|f\|_{L^2(w)}}\geq \frac{\big(N2^{N-2} w(S^2(f)\geq N2^{N-2})\big)^{1/2}}{\|f\|_{L^2(w)}}> \frac{e^{-2}}{48}\big([w]_{A_2}\ln (1+[w]_{A_2})\big)^{1/2}.$$
The proof is complete.
\end{proof}

It remains to establish the final ingredient.

\begin{lemma}\label{prob}
The inequality \eqref{limit} holds.
\end{lemma}
\begin{proof}
To study $\mathbb{P}(\tilde{S}^2(f)\geq N2^{N-2}|J_{2^N})$, let us start with a convenient lower bound for $\tilde{S}^2(f)$ on $J_{2^N}$. On this set, the process $(w,\sigma,f)$ was modified/extended $2^N$ times. Each iteration gives an \emph{independent} contribution to the size of the square function. Precisely, recall the discussion presented at the end of the previous subsection, and consider the sequence $\xi_1$, $\xi_2$, $\ldots$, $\xi_{2^N}$ of independent random variables with the identical distribution
$$ \mathbb{P}(\xi_k=4^{n-1})=4^{-n}\cdot \frac{3}{1-4^{-N}},\qquad n=1,\,2,\,\ldots,\,N.$$
Then we have
$$ \mathbb{P}(\tilde{S}^2(f)\geq N2^{N-2}|J_{2^N})\geq \mathbb{P}(\xi_1+\xi_2+\ldots+\xi_{2^N}\geq N2^{N-2})=\mathbb{P}\left(\eta_N\geq \frac{1}{4}\right),$$
where $\eta_N$ is the normalized sum
$$ \eta_N=\frac{\xi_1+\xi_2+\ldots+\xi_{2^N}}{N2^N}.$$
We will prove that $(\eta_N)_{N\geq 1}$ converges in distribution to $3/8$: this will imply that $\mathbb{P}(\eta_N\geq 1/4)\geq 1/3$ for sufficiently large $N$, which is precisely the claim. To prove the convergence, we will use L\'evy's theorem: we will show that for any $t\in \R$, the characteristic function $\psi_{\eta_N}(t)$ converges to $\exp(3it/8)$. To this end, note that
\begin{equation}\label{chpsi}
 \psi_{\eta_N}(t)=\left(\psi_{\xi_1}\left(\frac{t}{N2^N}\right)\right)^{2^N},
\end{equation}
where
$$ \psi_{\xi_1}(s)=\frac{3}{1-4^{-N}}\sum_{n=1}^N 4^{-n} \exp(4^{n-1}is)$$
is the characteristic function of $\xi_1$. We have the splitting
$$ \psi_{\xi_1}\left(\frac{t}{N2^N}\right)=I_1+I_2+I_3,$$
where
\begin{align*}
 I_1&=\frac{3}{1-4^{-N}}\sum_{n>N/2+\sqrt{\ln N}} 4^{-n}\exp\left(i\cdot \frac{4^{n-1}t}{N2^N}\right),\\
I_2&=\frac{3}{1-4^{-N}}\sum_{1\leq n\leq N/2+\sqrt{\ln N}} 4^{-n}\left[\exp\left(i\cdot \frac{4^{n-1}t}{N2^N}\right)-1-i\cdot \frac{4^{n-1}t}{N2^N}\right],\\
I_3&=\frac{3}{1-4^{-N}}\sum_{1\leq n\leq N/2+\sqrt{\ln N}} 4^{-n}\left(1+i\cdot \frac{4^{n-1}t}{N2^N}\right).
\end{align*}
Now, observe that 
$$ |I_1|\leq \frac{3}{1-4^{-N}}\sum_{n>N/2+\sqrt{\ln N}} 4^{-n}\leq \frac{4}{1-4^{-N}}\cdot 4^{-N/2-\sqrt{\ln N}}$$
and the last expression is of order $o(2^{-N})$ as $N\to \infty$. Next, using the elementary estimate $|e^{ix}-1-ix|\leq x^2$ for $x\in \R$, we get
$$ |I_2|\leq \frac{3}{1-4^{-N}}\sum_{1\leq n\leq N/2+\sqrt{\ln N}} 4^{-n}\cdot \left(\frac{4^{n-1}t}{N2^N}\right)^2\leq \frac{4^{N/2+\sqrt{\ln N}}-1}{1-4^{-N}}\cdot \frac{t^2}{N^2 4^N}$$ 
and, as before, the last quantity is of order $o(2^{-N})$ as $N\to \infty$. Finally, we have
\begin{align*}
 I_3&=\frac{1-4^{-\lfloor N/2+\sqrt{\ln N}\rfloor}}{1-4^{-N}}+\frac{3it}{4(1-4^{-N})}\cdot \frac{\lfloor N/2+\sqrt{\ln N}\rfloor}{N2^N}\\
 &=1+\frac{3it}{8(1-4^{-N})}\cdot \frac{1}{2^N}\\
&\qquad +\frac{4^{-N}-4^{-\lfloor N/2+\sqrt{\ln N}\rfloor}}{1-4^{-N}}+\frac{3it}{4(1-4^{-N})}\cdot \frac{\big(\lfloor N/2+\sqrt{\ln N}\rfloor-N/2\big)}{N2^N},
\end{align*}
and the last two summands are of order $o(2^{-N})$ as $N\to \infty$. Summarizing, we have shown that 
$$ \psi_{\xi_1}\left(\frac{t}{N2^N}\right)=1+\frac{r_N(t)}{2^N},$$
where $r_N(t)\to 3it/8$ as $N\to \infty$. Now the identity \eqref{chpsi} immediately gives $\psi_{\eta_N}(t)\to \exp(3it/8)$ and the claim follows.
\end{proof}

\section*{Artificial Intelligence Statement}
No AI tools were employed to obtain the above result.

\end{document}